\documentclass[11pt]{article}

\usepackage{amsmath}
\usepackage{amsthm}
\usepackage{amsfonts}
\usepackage{bbm}
\usepackage{amssymb}
\usepackage{graphicx}
\usepackage{color}

\usepackage{blindtext}
\usepackage{enumitem}

\usepackage{fancybox}
\newlength{\querylen}
\newcommand{\Cov}{\mathrm{Cov}}
\DeclareMathOperator{\Var}{\mathrm{Var}}
\newcommand{\fdc}{\overset{\mathrm{f.d.}}{\Rightarrow}}

\newcommand{\dy}{\mathrm{d} \mathit{y}}

\newcommand{\mmp}{\mathbb{P}}

\newcommand{\tp}{\overset{{\rm P}}{\to}}

\newcommand{\me}{\mathbb{E}}

\newcommand{\mr}{\mathbb{R}}
\newcommand{\mn}{\mathbb{N}}

\newcommand{\lin}{\lim_{n\to\infty}}

\newcommand{\lit}{\underset{t\to\infty}{\lim}}

\DeclareMathOperator{\1}{\mathbbm{1}}

\newtheorem{thm}{Theorem}[section]
\newtheorem{lemma}[thm]{Lemma}

\theoremstyle{definition}
\newtheorem{define}[thm]{Definition}
\theoremstyle{remark}
\newtheorem{rem}[thm]{Remark}

\begin{document}
\title{Weak convergence of random processes with immigration at random times}

\author{Congzao Dong\footnote{School of Mathematics and Statistics, Xidian University, 710126 Xi'an, P.R. China;\newline e-mail: czdong@xidian.edu.cn} \
and \ Alexander Iksanov\footnote{School of Mathematics and Statistics, Xidian University, 710126 Xi'an, P.R. China and Faculty of Computer Science and
Cybernetics, Taras Shevchenko National University of Kyiv, 01601
Kyiv, Ukraine;\newline e-mail: iksan@univ.kiev.ua}}

\maketitle
\begin{abstract}
\noindent By a random process with immigration at random times we mean a shot noise process with a random response function (response process) in which shots occur at arbitrary random times. The so defined random processes generalize random processes with immigration at the epochs of a renewal process which were introduced in [Iksanov et al. (2017). Bernoulli, 23, 1233--1278] and bear a strong resemblance to a random characteristic in general branching processes and the counting process in a fixed generation of a branching random walk generated by a general point process. We provide sufficient conditions which ensure weak convergence of finite-dimensional distributions of these processes to certain Gaussian processes. Our main result is specialised to several particular instances of random times and response processes.
\end{abstract}


\hphantom{2000 Mathematics Subject Classification: }Secondary: 60G55
\section{Introduction}\label{Sect1}

\subsection{Definition of random processes with immigration at random times}

Let $D:=D[0,\infty)$ be the Skorokhod space
of right-continuous real-valued functions which are defined on
$[0,\infty)$ and have finite limits from the left at each positive
point. Denoting, as usual, by $\mn_0:=\mn\cup\{0\}$ the set of
nonnegative integers, let $(T_k)_{k\in\mn_0}$ be a collection of
nonnegative, not necessarily ordered points such that
\begin{equation}\label{nt}
N(t):=\#\{k\in\mn_0: T_k\leq t\}<\infty\quad\text{a.s.\ for each}\quad t\geq 0.
\end{equation}
Although in most of applications the number of nonzero $T_k$'s is a.s.\ infinite (then $\lim_{k\to\infty}T_k=\infty$ a.s.\ is a sufficient condition for \eqref{nt}), the case of a.s.\ finitely many points is also allowed. Further, let $(X_j)_{j\in\mn}$ be independent copies of a random
process $X$ with paths in $D$ which vanishes on the negative
halfline. Finally, we assume that, for each $k\in\mn_0$, $X_{k+1}$
is independent of $(T_0, \ldots, T_k)$. In particular, the case of
complete independence of $(X_j)_{j\in\mn}$ and $(T_k)_{k\in\mn_0}$
is not excluded.

Put
$$Y(t):=\sum_{k\geq 0}X_{k+1}(t-T_k),\quad t\in\mr$$ (note that $Y(t)=0$ for $t<0$). We shall call $Y:=(Y(t))_{t\in\mr}$ {\it random process
with immigration at random times}. The interpretation is that
associated with the $k$th immigrant which arrives at time
$T_{k-1}$ is the random process $X_k$ which describes some model-dependent `characteristics' of the $k$th immigrant, for instance, $X_k(t-T_{k-1})$ may be the number of offspring of the immigrant at time $t$ or the fitness of the immigrant at time $t$. The value of $Y(t)$ is then given by the sum of `characteristics' of all immigrants that arrived up to and including time $t$.

\subsection{Pointers to earlier literature and relation of random processes with immigration at random times to other models}

When $(T_k)_{k\in\mn_0}$ is a zero-delayed standard random walk
with nonnegative jumps, that is, $T_0=0$ and
$(T_k-T_{k-1})_{k\in\mn}$ are independent identically distributed
nonnegative random variables, the random process $Y$ was called in
\cite{Iksanov+Marynych+Meiners:2017a} a {\it random process with
immigration at the epochs of a renewal process}. Thus, the set of
the latter processes constitutes a proper subset of the set of the
random processes with immigration at random times. We refer to \cite{Iksanov:2016} and \cite{Iksanov+Marynych+Meiners:2017a} for detailed surveys concerning earlier works on random processes with immigration at the epochs of a Poisson or renewal process. A non-exhaustive list of more recent contributions, not covered in the cited sources, includes \cite{Iksanov+Jedidi+Bouzzefour:2018}, \cite{Iksanov+Kabluchko:2018a}, \cite{Iksanov+Kabluchko:2018b},  \cite{Marynych+Verovkin:2017} and \cite{Pang+Taqqu:2019}.

Articles are relatively rare which focus on the random processes with immigration at random times other than renewal times. A selection of these can be traced via the references given in the recent article \cite{Pang+Zhou:2018}. The authors of \cite{Pang+Zhou:2018} investigate the random process of the form $$Y(t)=\sum_{k\geq 1}X_k(t-T_k)\1_{\{T_k\leq t\}},\quad t\geq 0,$$ where $X_k(t)=H(t,\eta_k)$ for $k\in\mn$, $H:[0,\infty)\times \mr^n\to \mr$ is a deterministic measurable function and $\eta_k$ is an $\mr^n$-valued random vector. Since $\eta_1$, $\eta_2,\ldots$ are assumed to be conditionally independent given $(T_j)_{j\in\mn}$ (rather than just independent), and $\eta_k$ is allowed to depend on $T_k$, the model in \cite{Pang+Zhou:2018} is slightly different from ours.

In \cite{Iksanov+Rashytov:2019}, another quite recent paper, functional limit theorems are proved for random processes with immigration at random times. There, the standing assumption is that $X$ is an eventually nondecreasing deterministic function which is regularly varying at $\infty$ of nonnegative index. We stress that the techniques used in the present work and in \cite{Iksanov+Rashytov:2019} are very different. 

Random processes with immigration at random times can be thought of as natural successors of two well-known branching processes: the general branching process (GBP) counted with random characteristic (see pp.~362-363 in \cite{Asmussen+Hering:1983}) and the counting process in a branching random walk (BRW). To define the GBP imagine a population initiated by a single ancestor at time $0$.
Denote by
\begin{itemize}

\item $\mathcal{T}$ a point process on $[0,\infty)$ describing the instants of time at which generic individual produces offspring;

\item $\Phi$ a random characteristic which is a random process on $\mr$ which vanishes on the negative halfline; the processes $\mathcal{T}$ and $\Phi$ are allowed to be arbitrarily dependent;

\item $J$ the collection of ever born individuals of the population.
\end{itemize}
Associated with each individual $n\in J$ is its birth time $\sigma_n$ and a random pair $(\mathcal{T}_n, \Phi_n)$, a copy of $(\mathcal{T}, \Phi)$. Furthermore, for different individuals these copies are independent. The GBP is given by $$Z(t):=\sum_{n\in J}\Phi_n(t-\sigma_n),\quad t\geq 0.$$ If $\Phi(t)=1$ for all $t\geq 0$, then $Z(t)$ is the total number of births up to and including time $t$. If $\Phi(t)=\1_{\{\tau>t\}}$ for a positive random variable $\tau$ interpreted as the lifetime of generic individual, then $Z(t)$ is the number of individuals alive at time $t$. More examples of this flavor can be found on p.~363 in \cite{Asmussen+Hering:1983}.

Consider now a BRW with positions of the $j$th generation individuals given by $(T(v))_{v\in\mathbb{V}_j}$ for $j\in\mn$, where $\mathbb{V}_j$ is the set of words of length $j$ over $\mn$ and for the individual $v\in \mathbb{V}_j$ its position on the real line is denoted by $T(v)$. Set $N_j(t):=\#\{v\in\mathbb{V}_j: T(v)\leq t\}$ for $t\in \mr$, so that $N_j(t)$ is the number of individuals in the $j$th generation of the BRW with positions $\leq t$. With the help of a branching property we obtain the basic decomposition
\begin{equation}\label{x}
N_j(t):=\sum_{v\in\mathbb{V}_{j-1}}N_1^{(v)}(t-T(v)),\quad t\in\mr,
\end{equation}
where $(N_1^{(v)}(t))_{t\geq 0}$ for $v\in\mathbb{V}_{j-1}$ are independent copies of $(N_1(t))_{t\geq 0}$ which are also independent of the $T(v)$, $v\in\mathbb{V}_{j-1}$. Motivated by an application to certain nested infinite occupancy schemes in a random environment the authors of the recent article \cite{Gnedin+Iksanov:2019} proved functional limit theorems in $D^\mn$ for $\Big(\frac{N_j(t\cdot)-a_j(t\cdot)}{b_j(t)}\Big)_{j\in\mn}$ with appropriate centering and normalizing functions $a_j$ and $b_j$. The standing assumption of \cite{Gnedin+Iksanov:2019} is that the positions $(T(v))_{v\in\mathbb{V}_1}$ are given by $(-\log P_k)_{k\in\mn}$, where $P_1$, $P_2,\ldots$ are positive random variables with an arbitrary joint distribution satisfying $\sum_{k\geq 1}P_k=1$ a.s.

\section{Main result}

Throughout the remainder of the paper we assume that $\me X(t)=0$ for all $t\geq
0$ and that the covariance
\begin{equation*}
f(u,w) := \Cov (X(u),X(w)) = \me X(u)X(w)
\end{equation*}
is finite for all $u,w \geq 0$. The variance of $X$ will be
denoted by $v$, that is, $v(t):=f(t,t)=\Var X(t)$.

Following \cite{Iksanov+Marynych+Meiners:2017a} we recall several notions related to regular variation in $\mr^2_+:=(0,\infty)\times (0,\infty)$. We refer to \cite{Bingham+Goldie+Teugels:1989} for an encyclopaedic treatment of regular variation on the positive halfline.
\begin{define}\label{regular_variation_R^2}
A function $r: [0,\infty)\times [0,\infty)\to\mr$ is {\it
regularly varying} in $\mr^2_+$ if there exists a function
$C: \mr^2_+ \to (0,\infty)$ such that
\begin{equation*}
\lit {r(ut,wt)\over r(t,t)}=C(u,w), \ \ u,w>0.
\end{equation*}
\end{define}
The function $C$ is called {\it limit function}. The definition
implies that $r(t,t)$ is  regularly varying at $\infty$, i.e.,
$r(t,t) \sim t^\beta \ell(t)$ as $t\to\infty$ for some $\ell$
slowly varying at $\infty$ and some $\beta\in\mr$ which is called
the {\it index of regular variation}. In particular,
$C(a,a)=a^{\beta}$ for all $a>0$ and further
$$C(au,aw)=C(a,a)C(u,w)=a^\beta C(u,w)$$ for all $a,u,w>0$.
\begin{define}\label{regular_variation_R^21}
A function $r: [0,\infty)\times [0,\infty)\to\mr$ will be called
{\it fictitious regularly varying} of index $\beta$ in $\mr^2_+$
if
\begin{equation*}
\lit {r(ut,wt)\over r(t,t)}=C(u,w),\quad u,w>0,
\end{equation*}
where $C(u,u):=u^\beta$ for $u>0$ and $C(u,w):=0$ for $u,w>0$,
$u\neq w$. A function $r$ will be called {\it wide-sense regularly
varying} of index $\beta$ in $\mr^2_+$ if it is either regularly
varying or fictitious regularly varying of index $\beta$ in
$\mr^2_+$.
\end{define}
The function $C$ corresponding to a fictitious regularly varying
function will also be called {\it limit function}.

The processes introduced in Definition \ref{definition_v_process} arise as weak limits in Theorem \ref{Prop:mu<infty} which is our main result. We
shall show that these are well-defined at the beginning of
Section \ref{pr}.
\begin{define}\label{definition_v_process}
Let $\rho>0$ and $C$ be the limit function for a wide-sense regularly varying
function (see Definition \ref{regular_variation_R^21}) in
$\mr^2_+$ of
index $\beta$ for some $\beta\in (-1,\infty)$. We shall denote by $V_{\beta,\rho}:=(V_{\beta,\rho}(u))_{u> 0}$ a
centered Gaussian process with the covariance
\begin{equation*}
\me V_{\beta,\rho}(u)V_{\beta,\rho}(w) =\int_0^{u\wedge w}
C(u-y, w-y) \,{\rm d}y^\rho=\rho \int_0^{u\wedge w}
C(u-y, w-y)y^{\rho-1} \, \dy, \quad u,w>0,
\end{equation*}
when $C(s,t) \neq 0$ for some $s,t>0$, $s \neq t$, and a centered
Gaussian process with independent values and variance $\me
V_{\beta,\rho}^2(u) =\rho {\rm B}(\beta+1, \rho) u^{\beta+\rho}$, otherwise. Here and hereafter, ${\rm B}(\cdot,\cdot)$ denotes the beta function.
\end{define}

Theorem \ref{Prop:mu<infty} given below is an extension of Proposition 2.1 in \cite{Iksanov+Marynych+Meiners:2017a} which treats the case where $(T_k)_{k\in\mn_0}$ is a zero-delayed ordinary random walk with positive increments. We shall write $Z_t(u)\fdc Z(u)$, $t\to\infty$ to denote weak convergence of finite-dimensional distributions, that is, for any $n\in\mn$ and any $0<u_1<u_2<\ldots<u_n<\infty$, $(Z_t(u_1),\ldots, Z_t(u_n))$ converges in distribution to $(Z(u_1),\ldots, Z(u_n))$, as $t\to\infty$. Also, as usual, $\tp$ denotes convergence in probability.
\begin{thm}\label{Prop:mu<infty}
Let finite $c,\rho>0$ and $\beta> -(\rho\wedge 1)$ be given. Assume that
\begin{itemize}
\item $v$ is a locally bounded function; $f(u,w) = \Cov(X(u), X(w))$ is a wide-sense regularly varying function of index $\beta$ in $\mr_+^2$ with limit function $C$;
\begin{equation}\label{0040}
\lit \underset{a\leq u\leq
b}{\sup}\,\bigg|{f(ut,(u+w)t)\over v(t)}-C(u,u+w)\bigg|=0
\end{equation}
for every $w>0$ and all $0<a<b<\infty$; when $f(u,w)$ is regularly varying, the function $u\mapsto C(u,u+w)$ is a.e.\ continuous on $(0,\infty)$ for every $w>0$;

\vspace{-0.25cm}
\item for all $y>0$
\begin{equation}    \label{eq:Lindeberg X-h}
v_y(t) := \me\Big(X^2(t)\1_{\{|X(t)|>y\sqrt{t^\rho v(t)}\}}\Big) = o(v(t)),\quad t\to\infty;
\end{equation}
\item
\begin{equation}\label{weak}
\sup_{y\in[0,\,T]}\Big|\frac{N(ty)}{t^\rho}-cy^\rho\Big|~\tp~0,\quad t\to\infty
\end{equation}
for all $T>0$;
\vspace{-0.25cm}
\item if $\beta\in (-(\rho\wedge 1), 0]$, then $\me N(t)<\infty$ for all $t\geq 0$ and
\begin{equation}\label{incr1}
\me (N(t)-N(t-1))=O(t^{\rho-1}),\quad t\to\infty.
\end{equation}
\end{itemize}
Then
\begin{equation}    \label{eq:1st summand convergence}
\frac{Y(ut)}{\sqrt{ct^\rho v(t)}} ~\fdc~ V_{\beta,\rho}(u), \quad
t\to\infty
\end{equation}
where $V_{\beta,\rho}$ is a centered Gaussian process introduced
in Definition \ref{definition_v_process}.
\end{thm}

\begin{rem}
The condition $\beta>-\rho$ is obviously needed to guarantee that the normalization $\sqrt{ct^\rho v(t)}$ diverges to $\infty$, as $t\to\infty$. Since $\me
V_{\beta,\,\rho}^2(u) =\rho {\rm B}(\beta+1, \rho)u^{\beta+\rho}$, the limit process $V_{\beta,\,\rho}$ is not well-defined unless $\beta>-1$.
\end{rem}

\begin{rem}\label{rem2}
Condition \eqref{weak} entails that the number of positive $T_k$'s is a.s.\ infinite. A simple sufficient condition for \eqref{weak} is
\begin{equation}\label{inter}
\lim_{t\to\infty} t^{-\rho} N(t)=c\quad\text{a.s.}
\end{equation}
Indeed, the latter entails $\lim_{t\to\infty} t^{-\rho} N(ty)=cy^\rho$ a.s.\ for each fixed $y\geq 0$. Furthermore, the convergence is locally uniform in $y$ a.s., that is, \eqref{weak} holds a.s. (hence, in probability) as the convergence of monotone functions to a continuous limit.

If $T_0<T_1<\ldots$ a.s., then a standard inversion procedure ensures that \eqref{inter} is equivalent to $\lim_{k\to\infty} k^{-1/\rho}T_k=c^{-1/\rho}$ a.s. If the collection $(T_k)_{k\in\mn_0}$ is not ordered or ordered in the nondecreasing (rather than increasing) order the aforementioned equivalence may fail to hold.
\end{rem}

\section{Applications}

In this section we discuss how Theorem \ref{Prop:mu<infty} reads for some particular $(T_k)_{k\in\mn_0}$ and $X$.

\subsection{Particular $(T_k)$}

\subsubsection{Perturbed random walks}\label{prw}

Let $(\xi_k, \eta_k)_{k\in\mn}$ be independent copies of a random
vector $(\xi, \eta)$ with positive arbitrarily dependent
components. Denote by $(S_k)_{k\in\mn_0}$ the zero-delayed ordinary random walk with nondegenerate at zero increments $\xi_k$, that
is, $S_0:=0$ and $S_k:=\xi_1+\ldots+\xi_k$ for $k\in\mn$. Consider
a {\it perturbed} random walk
\begin{equation}\label{perturbed}
T_k:=S_{k-1}+\eta_k,\quad k\in \mn.
\end{equation}
It is convenient to define the corresponding counting process on $\mr$ rather than on $[0,\infty)$, that is, $N(t)=\#\{k\in\mn: T_k\leq t\}$ for $t\in\mr$. Then, of course, $N(t)=0$ a.s.\ for $t<0$.

Condition \eqref{weak} holds for this particular $N(t)$ in view of Lemma \ref{appl_BS} in combination with Remark \ref{rem2}.
\begin{lemma}\label{appl_BS}
If $\mu:=\me \xi<\infty$, then
$\lim_{t\to\infty} t^{-1} N(t)=\mu^{-1}$ a.s.
\end{lemma}
\begin{proof}
Set $\nu(t):=\sum_{k\geq 0}\1_{\{S_k\leq t\}}$ for $t\geq 0$. For
$t>0$ and $y\in (0,t)$, the following inequalities hold with
probability one
\begin{equation}\label{ineq100}
\nu(t-y)-
\sum_{k=1}^{\nu(t)}\1_{\{\eta_k>y\}}=\sum_{k=1}^{\nu(t)}\1_{\{S_{k-1}\leq
t-y\}}- \sum_{k=1}^{\nu(t)}\1_{\{\eta_k>y\}}\leq
\sum_{k=1}^{\nu(t)}\1_{\{S_{k-1}+\eta_k\leq t\}}=N(t)\leq
\nu(t).
\end{equation}
By the strong law of large numbers for ordinary random walks $\lin
n^{-1}\sum_{k=1}^n\1_{\{\eta_k>y\}}=\me
\1_{\{\eta>y\}}=\mmp\{\eta>y\}$ a.s. Since $\lit \nu(t)=\infty$
a.s., it follows that $\lim_{t\to\infty}
\sum_{k=1}^{\nu(t)}\1_{\{\eta_k>y\}}/\nu(t)=\mmp\{\eta>y\}$ a.s.
Recall that $\lit t^{-1}\nu(t)=\mu^{-1}$ a.s.\ by the strong
law of large numbers for renewal processes, whence
$${\sum_{k=1}^{\nu(t)}\1_{\{\eta_k>y\}}\over t}={\sum_{k=1}^{\nu(t)}\1_{\{\eta_k>y\}}\over \nu(t)}{\nu(t)\over t}~\to~\frac{\mmp\{\eta>y\}}{\mu}\quad
\text{a.s.}$$ as $t\to\infty$. Hence, using \eqref{ineq100} we
infer that
$$\mu^{-1}-\mu^{-1}\mmp\{\eta>y\}\leq
\underset{t\to\infty}{\lim\inf}\,t^{-1}N(t)\leq
\underset{t\to\infty}{\lim\sup}\, t^{-1}N(t) \leq \mu^{-1}\quad\text{a.s.}$$ Letting $y\to\infty$ gives
$\lim_{t\to\infty} t^{-1} N(t)=\mu^{-1}$ a.s.
\end{proof}

To take care of the case when $\beta\in (-1,0)$ in Theorem \ref{Prop:mu<infty} we note that
\begin{equation}\label{equ}
\me N(t)=\me U(t-\eta)=\int_{[0,\,t]}U(t-y){\rm d} G(y),\quad t\in\mr
\end{equation}
where, for $t\in\mr$, $U(t):=\sum_{k\geq 0}\mmp\{S_k\leq t\}$ is
the renewal function and $G(t):=\mmp\{\eta\leq t\}$. In particular, by monotonicity and our assumption that $\mmp\{\xi=0\}<1$, $\me N(t)\leq U(t)<\infty$ for all $t\geq 0$. Further, condition \eqref{incr1} holds because subadditivity of $U$ on $\mr$ entails $0\leq \me (N(t)-N(t-1))\leq U(1)$.

\subsubsection{Non-homogeneous Poisson process}\label{nonhom}

Assume that $(N(t))_{t\geq 0}$ is a non-homogeneous Poisson process with mean function $m(t):=\me N(t)$ for $t\geq 0$ which satisfies $m(t)\sim c_0 t^{\rho_0}$ as $t\to\infty$ for some positive $c_0$ and $\rho_0$. We can identify the process $(N(t))_{t\geq 0}$ with the process $(\mathcal{P}(m(t)))_{t\geq 0}$, where $(\mathcal{P}(t))_{t\geq 0}$ is a homogeneous Poisson process of unit intensity. As a consequence of the strong law of large numbers for $\mathcal{P}(t)$ we obtain $\lim_{t\to\infty}t^{-\rho_0}N(t)=c_0 $ a.s. In view of Remark \ref{rem2} condition \eqref{weak} holds for the present $N(t)$ with $c=c_0$ and $\rho=\rho_0$. An additional assumption $m(t)-m(t-1)=O(t^{\rho_0-1})$ as $t\to\infty$ guarantees that condition \eqref{incr1} also holds.

\subsubsection{Positions in the $j$th generation of a branching random walk}

Consider a BRW generated by a point process with the points given by the successive positions of the same random walk $(S_n)_{n\geq 1}$ as in Section \ref{prw}. Assume that $\mu=\me\xi<\infty$. Denote by $(T_{k,j})_{k\in\mn}$, $j\in\mn$ the positions of the $j$th generation individuals and by $N_j(t)$, $j\in\mn$, $t\geq 0$, the number of the $j$th generation individuals with positions $\leq t$. In this example we identify $(T_k)_{k\in\mn_0}$ with $(T_{k,j})_{k\in\mn}$ for some integer $j\geq 2$, hence $N(t)$ with $N_j(t)$.

Set $U_j(t):=\me N_j(t)$ for $j\in\mn$ and $t\geq 0$. From the representation which is a counterpart of \eqref{x}
\begin{equation}\label{z}
N_j(t)=\sum_{k\geq 1}N_1^{(k)}(t-T_{k,j-1}),\quad t\geq 0
\end{equation}
where $(N_1^{(1)}(t))_{t\geq 0}$, $(N_1^{(2)}(t))_{t\geq 0},\ldots$ are independent copies
of $(N_1(t))_{t\geq 0}$ which are independent of $(T_{k,j-1})_{k\in\mn}$, we obtain $$U_j(t)=\int_{[0,\,t]}U_1(t-y){\rm d}U_{j-1}(y),\quad t\geq 0.$$
By the elementary renewal theorem, $U_1(t)=O(t)$ as $t\to\infty$. Further, by monotonicity, $U_j(t)\leq U_1(t)U_{j-1}(t)$ for $t\geq 0$ which shows that $U_j(t)<\infty$ for all $t\geq 0$ and that
\begin{equation}\label{y}
U_j(t)=O(t^j),\quad t\to\infty.
\end{equation}

To show that \eqref{incr1} holds we write by using subadditivity of $U_1(t)+1$ and monotonicity of $U_1(t)$
\begin{eqnarray*}
U_j(t)-U_j(t-1)&=&\int_{[0,\,t-1]}(U_1(t-y)-U_1(t-1-y)){\rm d}U_{j-1}(y)+\int_{(t-1,\,t]}U_1(t-y){\rm d}U_{j-1}(y)\\&\leq&
(U_1(1)+1)U_{j-1}(t-1)+U_1(1)(U_{j-1}(t)-U_{j-1}(t-1))\\&\leq& (U_1(1)+1)U_{j-1}(t).
\end{eqnarray*}
Invoking \eqref{y} proves \eqref{incr1} with $\rho=j$.

To check \eqref{weak} we assume for simplicity that $\sigma^2:=\Var \xi<\infty$ (this condition is by no means necessary but enables us to avoid some additional calculations). Theorem 1.3 in \cite{Iksanov+Kabluchko:2018a} entails that $$\frac{N_j(t\cdot)-(t\cdot)^j/(j!\mu^j)}{\sqrt{\sigma^2\mu^{-2j-1}t^{2j-1}}}$$ converges weakly to a $(j-1)$-times integrated Brownian motion in $D$ equipped with the $J_1$-topology. Of course, this immediately yields \eqref{weak} with $\rho=j$ and $c=(j!\mu^j)^{-1}$.

\subsection{Particular $X$}

Let $(\eta_k)_{k\in\mn}$ be independent copies of a random variable $\eta$ such that, for each $k\in\mn_0$, $\eta_{k+1}$ is independent of $(T_0,\ldots, T_k)$.

In Section 3 of \cite{Iksanov+Marynych+Meiners:2017a} it was checked that the covariance functions $f$ of the response processes $X$ discussed in parts (a), (b) and (e) below (parts (a) and (b) below) are regularly varying in $\mr^2_+$ of index $\beta$ (satisfy \eqref{0040}).

\noindent (a) Let $X(t)=\1_{\{\eta>t\}}-\mmp\{\eta>t\}$ with $\mmp\{\eta>t\}\sim t^\beta\ell(t)$ as $t\to\infty$ for some $\beta\in (-1,0)$. In this case, $C(u,w)=(u\vee w)^\beta$ for $u,w>0$, so that $C(u, u+w)=(u+w)^\beta$ is continuous in $u$ for every $w>0$. Further, $v(t)=\mmp\{\eta>t\}\mmp\{\eta\leq t\}$ is bounded. Finally, condition \eqref{eq:Lindeberg X-h} holds in view of $|X(t)|\leq 1$ a.s.

\noindent (b) Let $X(t)=\eta g(t)$, where $\me \eta=0$, $\Var \eta\in (0,\infty)$ and $g:[0,\infty)\to\mr$ varies regularly at $\infty$ of index $\beta/2$ for some $\beta>-1$ and $g\in D$. In this case, $C(u,w)=(uw)^{\beta/2}$ for $u,w>0$, so that $C(u, u+w)=(u(u+w))^{\beta/2}$ is continuous in $u$ for every $w>0$. Also, $v(t)=(\Var \eta)g^2(t)$ is locally bounded. Let $\rho>0$. Observe now that $\lim_{t\to\infty}(\sqrt{t^\rho v(t)}/|g(t)|)=\infty$
implies that, for all $y>0$, $$\me X^2(t)\1_{\{|X(t)|>y\sqrt{t^\rho v(t)}\}}=g^2(t)\me \eta^2\1_{\{|\eta|>y\sqrt{t^\rho v(t)}/|g(t)|\}}=o(v(t)),\quad t\to\infty,$$ that is, \eqref{eq:Lindeberg X-h} holds. The corresponding limit process admits a stochastic integral representation $$V_{\beta,\,\rho}(u)=\int_{[0,\,u]}(u-y)^{\beta/2}{\rm d}W(y^\rho),\quad u>0,$$ where $(W(u))_{u\geq 0}$ is a Brownian motion and $\beta>-(\rho\wedge 1)$.

\noindent (c) Let $X$ be a $D$-valued centered random process with finite second moments satisfying, for some interval $I\subset (0,\infty)$, $\me \sup_{s\in I}X^2(s)<\infty$.
Assume also it is self-similar of Hurst exponent $\beta/2$ for some $\beta>0$. By self-similarity, $v(t)=t^\beta \me X^2(1)$ (locally bounded function) and $$\frac{f(ut,wt)}{v(t)}=\frac{\me X(u)X(w)}{\me X^2(1)},\quad u, w>0$$ which shows that $f$ is regularly varying in $\mr^2_+$ of index $\beta$ with limit function $C(u,w)=(\me X(u)X(w))/(\me X^2(1))$ and that \eqref{0040} trivially holds. Continuity of $C(u, u+w)$ in $u>0$ for every $w>0$ is justified by the facts that, with probability one, $X(u)X(u+w)$ does not have fixed discontinuities and that $\me \sup_{s\in [a,b]}X^2(s)<\infty$ for all $0<a<b<\infty$ (use self-similarity) in combination with the Lebesgue dominated convergence theorem: for any deterministic $u>0$ $\lim_{s\to 0} X(u+s)X(u+s+w)=X(u)X(u+w)$ a.s.\ and for any  $s\in\mr$ sufficiently close to $0$ $|X(u+s)X(u+s+w)|\leq \sup_{v\in [a,\,b]}X^2(v)$ a.s.\ for large enough $b>0$ and small enough $a>0$. Finally, condition \eqref{eq:Lindeberg X-h} holds in view of $$\me X^2(t)\1_{\{|X(t)|>y\sqrt{t^\rho v(t)}\}}= t^\beta \me X^2(1)\1_{\{|X(1)|>(\me X^2(1))^{1/2}yt^{\rho/2}\}}=o(t^\beta),\quad t\to\infty,$$ where $\rho>0$.

In particular, if $X(t)=W(t^\beta)$ for $\beta>0$, where, as before, $(W(t))_{t\geq 0}$ is a Brownian motion, then, for any $\rho>0$, $V_{\beta,\,\rho}(u)=(\rho {\rm B}(\beta+1,\rho))^{1/2}W(u^{\beta+\rho})$ for $u\geq 0$.

\noindent (d) Let $X(t)=N(t)-\me N(t)=N(t)-m(t)$, where $(N(t))_{t\geq 0}$ is a non-homogeneous Poisson process with mean function $m(t)$ as discussed in Section \ref{nonhom}. In this case, $v(t)=m(t)\sim c_0 t^{\rho_0}$ as $t\to\infty$. Since $m(t)$ is a nondecreasing function, it must be locally bounded. For $u,v>0$, $f(u,v)=\me (N(u)-m(u))(N(v)-m(v))= m(u\wedge v)$. Hence, $f$ is regularly varying in $\mr_+^2$ of index $\rho_0$ with limit function $C(u,v)=(u\wedge v)^{\rho_0}$. Further, it is obvious that \eqref{0040} holds and that, for every $w>0$, $C(u, u+w)=u^{\rho_0}$ is continuous in $u$. It remains to check that condition \eqref{eq:Lindeberg X-h} holds. To this end, we use H\"{o}lder's inequality and then Markov's inequality to obtain, for $\rho, y>0$,
\begin{eqnarray*}
&&\me (N(t)-m(t))^2\1_{\{|N(t)-m(t)|>y\sqrt{t^\rho m(t)}\}}\\&\leq& \big(\me (N(t)-m(t))^4\big)^{1/2}\big(\mmp\{|N(t)-m(t)|>y\sqrt{t^\rho m(t)}\}\big)^{1/2}\\&\leq& (m(t)(1+3m(t)))^{1/2}y^{-1}t^{-\rho/2}=o(m(t))
\end{eqnarray*}
which proves \eqref{eq:Lindeberg X-h}.

The limit process $V_{\rho_0,\,\rho}$ is the same time-changed Brownian motion as in point (c) in which the role of $\beta$ is played by $\rho_0$.

To give a concrete specialization of Theorem \ref{Prop:mu<infty} let $Y(t)$ denote the number of the second generation individuals in a BRW generated by a non-homogeneous Poisson process $(N(t))_{t\geq 0}$ as above. Then $(Y(t))_{t\geq 0}$ is a random process with immigration at random times, for $Y(t)$ admits a representation similar to \eqref{z} in which we take $j=2$, replace $N_2(t)$ with $Y(t)$ and $N_1(t)$ with $N(t)$ and let $(T_{k,1})_{k\in\mn}$ denote the atoms of $(N(t))_{t\geq 0}$. We shall write $T_k$ for $T_{k,1}$. According to Theorem \ref{Prop:mu<infty} in combination with the discussion above and in Section \ref{nonhom} we have the following limit theorem with a random centering:
$$\frac{Y(ut)-\sum_{k\geq 1}m(ut-T_k)\1_{\{T_k\leq ut\}}}{c_0 (\rho_0{\rm B}(\rho_0+1, \rho_0))^{1/2} t^{\rho_0}} ~\fdc~W(u^{2\rho_0}), \quad
t\to\infty,$$ where $(W(u))_{u\geq 0}$ is a Brownian motion.

\noindent (e) Let $X(t)=(t+1)^{\beta/2}Z(t)$, where $\beta\in (-1,0)$ and $(Z(t))_{t\geq 0}$ is a stationary Ornstein-Uhlenbeck process with variance $1/2$. In this case, $f(u,w) =\me (X(u)X(w)) = 2^{-1}(u+1)^{\beta/2}(w+1)^{\beta/2}e^{-|u-w|}$ is fictitious regularly varying in $\mr^2_+$ of index $\beta$. Furthermore, condition \eqref{0040} holds, that is, for every $w>0$, $$\frac{f(ut, (u+w)t)}{v(t)}=\frac{(ut+1)^{\beta/2}((u+w)t+1)^{\beta/2}}{(t+1)^\beta}e^{-wt}$$ converges to $0$, as $t\to\infty$ uniformly in $u\in [a, b]$ for all $0<a<b<\infty$. This stems from the fact that while the first factor converges to $u^{\beta/2}(u+w)^{\beta/2}$ uniformly in $u\in [a,b]$, the second factor converges to zero and does not depend on $u$. Further, $v(t)=2^{-1}(t+1)^\beta$ is bounded. By stationarity, for each $t > 0$, $Z(t)$ has the same distribution as a random variable $\theta$ having the  normal distribution with zero mean and variance $1/2$. Hence, with $\rho>0$,
$$\me X^2(t)\1_{\{|X(t)|>y\sqrt{t^\rho v(t)}\}}= (t+1)^\beta \me \theta^2\1_{\{|\theta|>2^{-1/2}yt^{\rho/2}\}}=o(t^\beta),\quad t\to\infty,$$ that is, condition \eqref{eq:Lindeberg X-h} holds. For $\beta>-(\rho\wedge 1)$, the corresponding limit process $V_{\beta,\,\rho}$ is a centered Gaussian process with independent values.

\section{Proof of Theorem \ref{Prop:mu<infty}}\label{pr}

When $C(u,w)=0$ for all $u,w>0$, $u\neq w$, the process $V_{\beta,\,\rho}$ exists as a Gaussian process with independent values, see Definition \ref{definition_v_process}. Now we intend to show that the Gaussian process $V_{\beta,\,\rho}$
is well-defined in the complementary case when $C(u,w)>0$ for some $u,w>0$, $u\neq w$. To this end, we check that the function $\Pi(s,t)$
given by
$$\Pi(s,t) ~:=~ \int_0^{s\wedge t} C(s-y,t-y) \, {\rm d}y^\rho,\quad s,t>0$$
is finite and positive semidefinite, that is, for any $j\in\mn$,
any $\gamma_1,\ldots, \gamma_j\in\mr$ and any
$0<v_1<\ldots<v_j<\infty$
\begin{align}
0\leq \sum_{i=1}^j & \gamma_i^2 \Pi(v_i,\,v_i)+2\sum_{1\leq r<l\leq j} \gamma_r\gamma_l \Pi(v_r, v_l)   \notag \\
&= \sum_{i=1}^{j-1}
\int_{v_{i-1}}^{v_i}\bigg(\sum_{s=i}^j\gamma_s^2 C(v_s-y,v_s-y)
+2\sum_{i\leq r<l\leq j}\gamma_r\gamma_l C(v_r-y,v_l-y)\bigg) \, {\rm d}y^\rho \notag \\
&\hphantom{=} +\gamma_j^2 \int_{v_{j-1}}^{v_j} C(v_j-y, v_j-y)
\,{\rm d}y^\rho,\label{defi}
\end{align}
where $v_0:=0$. In view of
\begin{equation}    \label{eq:impo ineq1}
|f(s,t)| \leq 2^{-1}(v(s)+v(t)), \quad s,t\geq 0,
\end{equation}
we infer
\begin{equation}    \label{eq:impo ineq2}
C(s-y,t-y) \leq 2^{-1} ((s-y)^\beta+(t-y)^\beta).
\end{equation}
Since $\beta>-1$ by assumption the latter ensures
$\Pi(s,t)<\infty$ for all $s,t>0$. Now we pass to the proof of
\eqref{defi}. Since the second term on the right-hand side of
\eqref{defi} is nonnegative, it suffices to prove that so is the
first. The function $C(s,t)$, $s,t>0$ is positive semidefinite as
a limit of positive semidefinite functions. Hence, for each $1\leq
i\leq j-1$ and $y\in (v_{i-1},\,v_i)$,
$$\sum_{s=i}^j\gamma_s^2 C(u_s-y, u_s-y)+2\sum_{i\leq r<l\leq
j}\gamma_r\gamma_l C(u_r-y,u_l-y)\geq 0.$$ Thus, the process
$V_{\beta,\,\rho}$ does exist as a Gaussian process with
covariance function $\Pi(s,t)$, $s,t>0$.

\begin{proof}[Proof of Theorem \ref{Prop:mu<infty}]
We treat simultaneously the case when $C(u,w)>0$ for some $u,w>0$, $u\neq w$ and the complementary case.

According to the Cram\'{e}r-Wold device relation \eqref{eq:1st
summand convergence} is equivalent to
\begin{equation}    \label{eq:Cramer-Wold device}
\frac{\sum_{i=1}^j \alpha_i \sum_{k \geq
0}X_{k+1}(u_it-T_k)\1_{\{T_k \leq u_it\}}}{\sqrt{ct^\rho v(t)}}
~\stackrel{\mathrm{d}}{\to}~ \sum_{i=1}^j \alpha_i
V_{\beta,\rho}(u_i)
\end{equation}
for all $j\in\mn$, all $\alpha_1,\ldots, \alpha_j\in\mr$ and all
$0<u_1<\ldots<u_j<\infty$. Here and hereafter, $\stackrel{\mathrm{d}}{\to}$ denotes convergence in distribution.
Define the $\sigma$-algebras $\mathcal{F}_0:=\sigma(T_0)$
and $\mathcal{F}_k:=\sigma(T_0, X_1, T_1, \ldots, X_k, T_k)$ for
$k\in\mn$ and set $\me_k(\cdot):=\me (\cdot|\mathcal{F}_k)$,
$k\in\mn_0$. Now observe that
$$\me_k \sum_{i=1}^j\alpha_i X_{k+1}(u_it-T_k)\1_{\{T_k\leq
u_i t\}}= 0,\quad k\in\mn_0$$ which shows that $\sum_{k\geq
0}\sum_{i=1}^j\alpha_i X_{k+1}(u_it-T_k)\1_{\{T_k\leq u_i t\}}$ is
a martingale limit.  In view of this, in order to prove
\eqref{eq:Cramer-Wold device}, one may use the martingale central
limit theorem (Corollary 3.1 in \cite{Hall+Heyde:1980}). The
theorem tells us that it suffices to verify
\begin{align}       \label{eq:mgale CLT1}
\sum_{k\geq 0} \me_k (Z_{k+1,\,t}^2) ~\stackrel{\mmp}{\to}~
D_{\beta,\rho}(u_1,\ldots, u_j),\quad t\to\infty
\end{align}
and
\begin{equation}    \label{eq:mgale CLT2}
\sum_{k\geq 0} \me_k
\big(Z_{k+1,\,t}^2\1_{\{|Z_{k+1,\,t}|>y\}}\big)
~\stackrel{\mmp}{\to}~ 0,\quad t\to\infty
\end{equation}
for all $y>0$, where
\begin{equation*}
Z_{k+1,\,t} := \frac{\sum_{i=1}^j \alpha_i \1_{\{T_k\leq
u_it\}}X_{k+1}(u_it-T_k)}{\sqrt{ct^\rho v(t)}}, \quad k\in\mn_0,\
t>0.
\end{equation*}

\noindent {\it Proof of \eqref{eq:mgale CLT1}}. We start by
writing
\begin{eqnarray*}
\sum_{k\geq 0} \me_k (Z_{k+1,\,t}^2) & = &
\frac{\sum_{i=1}^j \alpha_i^2 \sum_{k \geq 0}\1_{\{T_k\leq u_it\}}v(u_it-T_k)}{ct^\rho v(t)}  \\
& & + \frac{2\sum_{1\leq r<l\leq j}\alpha_r\alpha_l \sum_{k\geq
0}\1_{\{T_k\leq u_rt\}}f(u_rt-T_k, u_lt-T_k)}{ct^\rho v(t)}.
\end{eqnarray*}
We shall prove that, as $t\to\infty$,
\begin{equation}    \label{eq:convergence at u_1}
\frac{\sum_{k\geq 0} \1_{\{T_k\leq u_it\}}v(u_it-T_k)}{ct^\rho
v(t)} ~=~ \frac{\int_{[0,\,u_i]} v(t(u_i-y)) \, {\rm
d}N(ty)}{ct^\rho v(t)} ~\stackrel{\mmp}{\to}~ \rho {\rm
B}(\beta+1, \rho) u_i^{\beta+\rho}
\end{equation}
for all $1\leq i\leq j$ and that
\begin{eqnarray}    \label{eq:convergence at u_1,u_2}
\frac{\sum_{k\geq 0}\1_{\{T_k\leq
u_rt\}}f(u_rt-T_k,u_lt-T_k)}{ct^\rho v(t)} & = &
\frac{\int_{[0,\,u_r]}f(t(u_r-y),t(u_l-y)) \, {\rm d}N(ty)}{ct^\rho v(t)}    \notag  \\
& \stackrel{\mmp}{\to} & \int_0^{u_r} C(u_r-y,u_l-y){\rm d}y^\rho
\end{eqnarray}
for all $1\leq r<l\leq j$.

Fix any $u_r<u_l$ and pick $\varepsilon \in (0,u_r\wedge 1)$. 
We claim that, as $t\to\infty$,
\begin{equation}\label{inte01}
\int_{[0,\,u_r-\varepsilon]}{v(t(u_r-y))\over v(t)} \, {\rm
d}{N(ty)\over ct^\rho}\quad \overset{\mathbb{P}}{\to}\quad
\int_0^{u_r-\varepsilon}(u_r-y)^\beta \,{\rm d}y^\rho 
\end{equation}
and
\begin{equation}\label{inte02}
\int_{[0,\,u_r-\varepsilon]} {f(t(u_r-y),t(u_l-y))\over v(t)} \,
{\rm d} {N(ty)\over ct^\rho}\quad \overset{\mmp}{\to}\quad
\int_0^{u_r-\varepsilon}C(u_r-y, u_l-y){\rm d}y^\rho.
\end{equation}

To prove these limit relations we need some preparation. For each
$t>0$, the random function $\mathcal{G}_t$ defined by
$\mathcal{G}_t(y):=0$ for $y<0$, $:=N(ty)/N(tu_r)$ for
$y\in[0,u_r)$, and $=1$ for $y\geq u_r$ is a random distribution
function. Similarly, the function $\mathcal{G}$ defined by
$\mathcal{G}(y):=0$ for $y<0$, $:=(y/u_r)^\rho$ for $y\in[0,u_r)$,
and $=1$ for $y\geq u_r$ is a distribution function. 
According to \eqref{weak}, for every sequence $(t_n)_{n\in\mn}$ there exists a
subsequence $(t_{n_s})_{s\in\mn}$ such that $\lim_{s\to\infty}
t_{n_s}^{-\rho} N(t_{n_s}y)= cy^\rho$ a.s.\ for each $y\in [0,u_r]$. We would like to stress that uniformity of the convergence in \eqref{weak} ensures that
the subsequence $(t_{n_s})_{s\in\mn}$ does not depend on $y$ (without the uniformity assumption we should have taken a new subsequence $(t_{n_s})_{s\in\mn}$ for each particular $y\in [0,\,u_r]$; this would not be sufficient for what follows). The last limit relation guarantees $\lim_{s\to\infty}
N(t_{n_s}y)/N(t_{n_s}u_r)= (y/u_r)^\rho$ a.s.\ for each $y\in [0,u_r]$. Therefore, as $s\to\infty$ $\mathcal{G}_{t_{n_s}}$ converges weakly to
$\mathcal{G}$ with probability one.

\noindent {\sc Proof of \eqref{inte01}}. Write
\begin{eqnarray*}
&&\Big|\int_{[0,\,u_r-\varepsilon]}\frac{v(t_{n_s}(u_r-y))}{v(t_{n_s})}{\rm
d}\mathcal{G}_{t_{n_s}}(y)-\int_{[0,\,u_r-\varepsilon]}(u_r-y)^\beta
{\rm d}\mathcal{G}(y) \Big|\\&\leq&
\int_{[0,\,u_r-\varepsilon]}\Big|\frac{v(t_{n_s}(u_r-y))}{v(t_{n_s})}-(u_r-y)^\beta
\Big|{\rm
d}\mathcal{G}_{t_{n_s}}(y)\\&+&\Big|\int_{[0,\,u_r-\varepsilon]}(u_r-y)^\beta
{\rm
d}\mathcal{G}_{t_{n_s}}(y)-\int_{[0,\,u_r-\varepsilon]}(u_r-y)^\beta
{\rm d}\mathcal{G}(y)\Big|.
\end{eqnarray*}
By the uniform convergence theorem for regularly varying functions
(Theorem 1.5.2 in \cite{Bingham+Goldie+Teugels:1989}),
\begin{equation}\label{inte03}
\lim_{t \to \infty} \frac{v(t(u_r-y))}{v(t)} = (u_r-y)^\beta
\end{equation}
uniformly in $y \in [0,\,u_r-\varepsilon]$. This implies that the
first summand on the right-hand side of the penultimate centered
formula converges to $0$ a.s.\ as $s\to\infty$. The second summand
does so by the following reasoning. The function $g$ defined by
$g(y):=(u_r-y)^\rho$ for $y\in [0, u_r-\varepsilon]$ and $:=0$ for
$y>u_r-\varepsilon$ is bounded with one discontinuity point. With
this at hand it remains to invoke the aforementioned weak
convergence with probability one and the fact that $\mathcal{G}$
is a continuous distribution function. This implies that the
left-hand side of the penultimate centered formula with $t$
replacing $t_{n_s}$ converges in probability to $0$ as
$t\to\infty$. Multiplying it by $N(tu_r)/(ct^\rho)$ which
converges to $u_r^\rho$ in probability as $t\to\infty$ we arrive
at \eqref{inte01}.

\noindent {\sc Proof of \eqref{inte02}} is analogous. Instead of
\eqref{inte03} one has to use the following relation which is a
consequence of \eqref{0040}:
\begin{equation*} \lim_{t \to \infty} \frac{f(t(u_r-y),t(u_l-y))}{v(t)} = C(u_r-y, u_l-y)
\end{equation*}
uniformly in $y\in [0,u_r-\varepsilon]$. The role of $g$ is now
played by $g^\ast(y):=C(u_r-y, u_l-y)$ for $y\in [0,
u_r-\varepsilon]$ and $:=0$ for $y>u_r-\varepsilon$. In view of
\eqref{eq:impo ineq2}, this function is bounded. Also, $g^\ast$ is
a.e.\ continuous by assumption which in combination with the
absolute continuity of $\mathcal{G}$ is enough for completing the
proof of \eqref{inte02}.

As $\varepsilon\to 0+$, the right-hand sides of \eqref{inte01} and
\eqref{inte02} converge to $\int_0^{u_r}(u_r-y)^\beta \,{\rm
d}y^\rho=\rho {\rm B}(\beta+1, \rho)u_r^{\beta+\rho}$ and
$\int_0^{u_r} C(u_r-y,u_l-y){\rm d}y^\rho$, respectively. Thus,
relations \eqref{eq:convergence at u_1} and \eqref{eq:convergence
at u_1,u_2} are valid if we can show (see Theorem 4.2 in
\cite{Billingsley:1968}) that
\begin{equation}\label{aux1}
\lim_{\varepsilon\to 0+} \limsup_{t\to\infty} \mmp\bigg\{
\frac{\int_{(u_r-\varepsilon, \, u_r]} v(t(u_r-y)) \, {\rm
d}N(ty)}{ct^\rho v(t)} > \delta\bigg\} = 0
\end{equation}
and
\begin{equation}\label{aux2}
\lim_{\varepsilon\to 0+} \limsup_{t\to\infty} \mmp\bigg\{
\frac{\big|\int_{(u_r-\varepsilon, \, u_r]}f(t(u_r-y),t(u_l-y)) \,
{\rm d}N(ty) \big|}{ct^\rho v(t)}>\delta\bigg\}=0
\end{equation}
for any $\delta>0$.

Using \eqref{eq:impo ineq1} we obtain
\begin{align}
\int_{(u_r-\varepsilon, \, u_r]} & |f(t(u_r-y), t(u_l-y))| \, {\rm d}N(ty)\notag   \\
& \leq~ 2^{-1}\Big(\int_{(u_r-\varepsilon, \, u_r]} v(t(u_r-y))
\, {\rm d}N(ty)+ \int_{(u_r-\varepsilon, \, u_r]} v(t(u_l-y)) \, {\rm d}N(ty)\Big)\label{aux00}
\end{align}
which shows that a proof of \eqref{aux2} includes that of \eqref{aux1}. Therefore, we shall only prove \eqref{aux2}.

We first treat the second summand on the right-hand side of \eqref{aux00}. Since $$\lim_{t\to\infty}\frac{v(t(u_l-y))}{v(t)}=(u_l-y)^\beta$$ uniformly in $y\in (u_r-\varepsilon, u_r]$ (recall that $u_r<u_l$) we can use the argument given after formula \eqref{inte03} to conclude that $$\frac{\int_{(u_r-\varepsilon, \, u_r]} v(t(u_l-y)) \, {\rm d}N(ty)}{ct^\rho v(t)}~\tp~\int_{(u_r-\varepsilon, \,u_r]}(u_l-y)^\beta {\rm d}y^\rho,\quad t\to\infty.$$ The right-hand side converges to zero as $\varepsilon\to 0+$.

Now we are passing to the analysis of the first summand on the right-hand side of \eqref{aux00}. According to Potter's bound (Theorem 1.5.6 (iii) in
\cite{Bingham+Goldie+Teugels:1989}), for any chosen $A>1$,
$\gamma\in (0,\beta)$ when $\beta>0$ and $\gamma\in (0,\beta+1)$ when $\beta\in
(-(\rho\wedge 1),0]$ there exists $t_0>0$ such that
$$\frac{v(t(u_r-y))}{v(t)}\leq A (u_r-y)^{\beta-\gamma}$$ whenever $t\geq t_0$ and $t(u_r-y)\geq
t_0$. Then, for $t\geq t_0/\varepsilon$,
\begin{eqnarray}
&&\frac{\int_{(u_r-\varepsilon, \, u_r]} v(t(u_r-y)) \, {\rm
d}N(ty)}{ct^\rho v(t)}\notag\\&\leq& \frac{\int_{(u_r-\varepsilon, \,
u_r-t_0/t]} v(t(u_r-y)) \, {\rm d}N(ty)}{ct^\rho
v(t)}+\frac{\int_{(u_r-t_0/t, \, u_r]} v(t(u_r-y)) \, {\rm
d}N(ty)}{ct^\rho v(t)}\notag\\&\leq& \frac{A \int_{(u_r-\varepsilon, \,
u_r-t_0/t]}(u_r-y)^{\beta-\gamma}{\rm d}N(ty)}{ct^\rho}+
\frac{(N(tu_r)-N(tu_r-t_0))\sup_{x\in [0,\,t_0]} v(x)}{ct^\rho v(t)}.\label{aux0}
\end{eqnarray}
We claim that the second term on the right-hand side in \eqref{aux0} converges to zero in probability as $t\to\infty$. For the proof we first note that the function $v$ is locally bounded by assumption. With this at hand, the claim follows from \eqref{incr1} in combination with Markov's inequality when $\beta\in (-(\rho\wedge 1), 0)$ or $\beta=0$ and ${\lim\inf}_{t\to\infty}v(t)=0$ and from $t^{-\rho}(N(t)-N(t-t_0))\tp 0$ as $t\to\infty$ which, in its turn, is a consequence of \eqref{weak} when $\beta>0$ or $\beta=0$ and ${\lim\inf}_{t\to\infty}v(t)>0$.

While treating the first summand on the right-hand side in \eqref{aux0} we consider two cases separately.

\noindent {\sc Case $\beta>0$} in which $\beta-\gamma>0$.
The first summand is bounded from above by $A\varepsilon^{\beta-\gamma}N(tu_r)/(ct^\rho)$ which converges to $A\varepsilon^{\beta-\gamma}u_r^\rho$ in probability as $t\to\infty$. Therefore, for any $\delta>0$, $$\limsup_{t\to\infty} \mmp\{A\varepsilon^{\beta-\gamma}N(tu_r)/(ct^\rho) > \delta\} \leq \1_{[0,A\varepsilon^{\beta-\gamma}u_r^\rho]}(\delta).$$ It remains to note that the right-hand side converges to zero as $\varepsilon\to 0+$.

\noindent {\sc Case $\beta\in (-(\rho\wedge 1), 0]$} in which $\beta-\gamma<0$. Invoking Markov's inequality we see that it suffices to prove that
\begin{equation}    \label{eq:Markov's inequality at u_1}
\lim_{\varepsilon\to 0+} \limsup_{t\to\infty}
\frac{\int_{(u_r-\varepsilon, \, u_r]}(u_r-y)^{\beta-\gamma}\, {\rm
d}L(ty)}{t^\rho} = 0,
\end{equation}
where $L(t):=\me N(t)$ for $t\geq 0$.

Write, for large enough $t$, positive constants $C_1$ and $C_2$, and $i=1,2$
\begin{eqnarray*}
\int_{(u_r-\varepsilon, u_r]}(u_r-y)^{\beta-\gamma}\, {\rm
d}L(ty)&\leq& \sum_{k=0}^{[\varepsilon t]}\int_{(u_r-t^{-1}(k+1), u_r-t^{-1}k]}(u_r-y)^{\beta-\gamma}\, {\rm
d}L(ty)\\&\leq& \sum_{k=0}^{[\varepsilon t]}(k/t)^{\beta-\gamma}(L(tu_r-k)-L(tu_r-(k+1)))\\&\leq&
\begin{cases}
C_1t^{-(\beta-\gamma)}\sum_{k=0}^{[\varepsilon t]}k^{\beta-\gamma}(tu_r-k)^{\rho-1}, &   \text{if } \ \rho\geq 1,   \\
C_2t^{-(\beta-\gamma)}\sum_{k=0}^{[\varepsilon t]}k^{\beta-\gamma}(tu_r-k+1)^{\rho-1}, & \text{if} \ \rho\in (0,1),
\end{cases}
\\&\leq& C_it^{-(\beta-\gamma)}\sum_{k=1}^{[\varepsilon t]}\int_{k-1}^k y^{\beta-\gamma}(tu_r-y)^{\rho-1}{\rm d}y \\&\leq& C_i t^{-(\beta-\gamma)}\int_0^{\varepsilon t} y^{\beta-\gamma}(tu_r-y)^{\rho-1}{\rm d}y\\&=&C_i t^\rho \int_0^\varepsilon y^{\beta-\gamma}(u_r-y)^{\rho-1}{\rm d}y,
\end{eqnarray*}
where the third inequality is a consequence of \eqref{incr1}, and we take $i=1$ when $\rho\geq 1$ and $i=2$ when $\rho\in (0,1)$. This proves \eqref{eq:Markov's inequality at u_1}, and \eqref{eq:mgale CLT1} follows.

\noindent {\it Proof of \eqref{eq:mgale CLT2}}: The following inequality holds for real $a_1,\ldots,a_m$
\begin{eqnarray}\label{for_ref}
(a_1+\ldots+a_m)^2\1_{\{|a_1+\ldots+a_m|>y\}}&\leq&
(|a_1|+\ldots+|a_m|)^2\1_{\{|a_1|+\ldots+|a_m|>y\}}\notag\\&\leq&
m^2 (|a_1| \vee\ldots\vee |a_m|)^2\1_{\{m(|a_1| \vee\ldots\vee
|a_m|)>y\}}\notag\\&\leq&
m^2\big(a_1^2\1_{\{|a_1|>y/m\}}+\ldots+a_m^2\1_{\{|a_m|>y/m\}}\big).
\end{eqnarray}
This in combination with the regular variation of $t^\rho v(t)$ guarantees it is sufficient to show
that
\begin{equation}\label{Proof_of_Prop21_conv_to_0}
\sum_{k\geq 0} \! \1_{\{T_k \leq t\}} \me_k \!
\bigg(\frac{X^2_{k+1}(t-T_k)}{t^\rho v(t)}
\!\1_{\{|X_{k+1}(t-T_k)|>y\sqrt{t^\rho v(t)}\}}\bigg)
~\stackrel{\mmp}{\to}~ 0
\end{equation}
for all $y>0$.

By Proposition 1.5.8 in \cite{Bingham+Goldie+Teugels:1989}, $t^\rho v(t)\sim (\rho+\beta)\int_0^t y^{\rho-1}v(y){\rm d}y$ as $t\to\infty$. Therefore, while proving Theorem \ref{Prop:mu<infty} we can interchangeably use $t^\rho v(t)$ or $(\rho+\beta)\int_0^t y^{\rho-1}v(y){\rm d}y$ in the denominator of \eqref{eq:1st summand convergence}. Therefore, without loss of generality we can and do assume that $t^\rho v(t)$ is nondecreasing, for so is its asymptotic equivalent. Thus, relation
\eqref{Proof_of_Prop21_conv_to_0} follows if we can prove that
\begin{equation}    \label{eq:renewal asymptotics}
\frac{1}{t^\rho v(t)}\int_{[0,\,t]} v_y(t-x) \, {\rm
d}N(x)~\tp~ 0,\quad t\to\infty
\end{equation}
for all $y>0$.

Fix any $y>0$. Formula \eqref{eq:Lindeberg X-h} ensures that given $\varepsilon>0$ there exists $t_0>0$ such that $v_y(t)\leq \varepsilon v(t)$ whenever $t\geq t_0$. With this at hand we obtain
\begin{eqnarray*}
\frac{1}{t^\rho v(t)}\int_{[0,\,t]} v_y(t-x) \, {\rm
d}N(x)&=&\frac{1}{t^\rho v(t)}\Big(\int_{[0,\,t-t_0]} v_y(t-x) \, {\rm
d}N(x)+\int_{(t-t_0,\,t]} v_y(t-x) \, {\rm
d}N(x)\Big)\\&\leq& \frac{\varepsilon}{t^\rho v(t)}\int_{[0,\,t]} v(t-x) \, {\rm
d}N(x)\\&+&\frac{(N(t)-N(t-t_0))\sup_{x\in [0,\,t_0]}v_y(x)}{t^\rho v(t)}.
\end{eqnarray*}
Using \eqref{eq:convergence at u_1} with $u_i=1$ and denoting the first summand on the right-hand side by $J(t,\varepsilon)$ we conclude that, for any $\delta>0$, $$\lim_{\varepsilon\to 0+}\limsup_{t\to\infty} \mmp\{J(t,\varepsilon)> \delta\}=0.$$  Since $v_y(t)\leq v(t)$ for all $t\geq 0$, and $v$ is locally bounded by assumption, so is $v_y$. Therefore, the second summand on the right-hand side converges to zero in probability as $t\to\infty$ by the same reasoning as given for the second summand on the right-hand side of \eqref{aux0}.

The proof of Theorem \ref{Prop:mu<infty} is complete.
\end{proof}

\vspace{5mm}

\noindent {\bf Acknowledgement}. The support of the Chinese Belt and Road Programme DL20180077 is gratefully acknowledged.

\end{document}